\documentclass[10pt]{amsart}
\usepackage{amsmath}
\usepackage{amssymb}

\theoremstyle{plain}
\newtheorem{theorem}{\bf Theorem}[section]
\newtheorem{proposition}[theorem]{\bf Proposition}
\newtheorem{lemma}[theorem]{\bf Lemma}

\theoremstyle{definition}

\newtheorem{definition}[theorem]{\bf Definition}

\newcommand{\N}{\mathbb N}
\newcommand{\Z}{\mathbb Z}

 \DeclareMathOperator{\ind}{ind}
 \DeclareMathOperator{\ord}{ord}

\renewcommand{\time}{\negthinspace \times \negthinspace}

\numberwithin{equation}{section}

\begin{document}
\title[Minimal zero-sum sequences of length five]{Minimal zero-sum sequences of length five over finite cyclic groups}

\author{Jiangtao Peng and Yuanlin Li*}
\address{College of Science, Civil Aviation University of China, Tianjin 300300, P.R. China} \email{jtpeng1982@yahoo.com.cn}

\address{Department of Mathematics, Brock University, St. Catharines, Ontario,
Canada L2S 3A1} \email{yli@brocku.ca}

\thanks{This research was supported in part by a Discovery
Grant from the Natural Sciences and Engineering Research Council of
Canada,  the National Science Foundation of China (Grant Nos. 11126137 and 11271250) and a research grant from  Civil Aviation University of China (No. 2010QD02X).\\
*Corresponding author: Yuanlin Li, Department of Mathematics, Brock
University, St. Catharines, Ontario Canada L2S 3A1, Fax:(905)
378-5713;\\ E-mail: yli@brocku.ca (Y. Li) \\ \today}

\begin{abstract}
Let $G$ be a finite cyclic group. Every sequence $S$ of length $l$ over $G$ can be written in the form
$S=(n_1g)\cdot\ldots\cdot(n_lg)$ where $g\in G$ and $n_1, \ldots, n_l\in[1, \ord(g)]$, and the index $\ind(S)$
 of $S$ is defined to be the minimum of $(n_1+\cdots+n_l)/\ord(g)$ over all possible $g\in G$
 such that $\langle g \rangle =G$. In this paper, we determine the index of any minimal zero-sum sequence $S$ of length 5 when  $G=\langle g\rangle$ is a cyclic group of a prime order and  $S$ has the form $S=g^2(n_2g)(n_3g)(n_4g)$. It is shown that if $G=\langle g\rangle$ is a cyclic group of  prime order $p \geq 31$, then every minimal zero-sum sequence $S$ of the above mentioned form has index 1 except in the case that $S=g^2(\frac{p-1}{2}g)(\frac{p+3}{2}g)((p-3)g)$.
\end{abstract}
\maketitle

\section {Introduction}
\medskip
Throughout the paper $G$ is assumed to be a finite cyclic
group of order $n$ written additively.  Denote by  $\mathcal{F}(G)$, the free abelian monoid with basis $G$ and
elements of $\mathcal{F}(G)$ are called \emph{sequences} over $G$. A sequence of length $ l$ of
not necessarily distinct elements from $G$ can be written in the
form $S=(n_1g)\cdot \, \ldots \,\cdot (n_lg)$ for some $g\in G$. Call $S$  a \emph{zero-sum
sequence} if the sum of $S$ is zero (i.e. $\sum_{ i=1}^l n_ig=0$). If $S$ is a zero-sum sequence, but no proper
nontrivial subsequence of $S$ has sum zero, then  $S$ is called a \emph{minimal zero-sum sequence}.  Recall that the index of a sequence $S$ over $G$ is defined as follows.
\begin{definition}
 For a sequence over $G$
 $$S=(n_1g)\cdot\ldots\cdot(n_lg), \,\,\, \mbox{where} \,\, 1\le n_1, \ldots, n_l \le \ord(g),$$
 the index of $S$ is defined by
 $\ind(S)=\min\{\| S \|_g \,|\,g\in G \, \mbox{with}\,\, G=\langle g \rangle\}$ where  $$\|S\|_g=\frac{n_1+\cdots+n_l}{\ord(g)}.$$
\end{definition}
\noindent Clearly, $S$ has sum zero if and only if $\ind(S)$ is an integer. There are also slightly different definitions of the index in the literature, but they are all equivalent (see Lemma 5.1.2 in \cite{Ge:09a}).

 The index of a sequence is a crucial invariant
in the investigation of (minimal) zero-sum sequences (resp. of
zero-sum free sequences) over cyclic groups. It was first addressed
by Kleitman-Lemke (in the conjecture \cite[page 344]{KL:89}),
used as a key tool by Geroldinger (\cite[page 736]{G:87}), and
then investigated by Gao \cite{Gao:00} in a systematical way. Since then it has received a great deal of attention (see for example \cite{CFS:99, CS:05, GG:09, GLPPW:11, Ge:09a, Ge-HK06a, LPYZ:10,
P:04, SC:07, XY:09, Y:09}).

 A main focus of the investigation of index is to determine minimal zero-sum sequences of index 1. If $S$ is a minimal zero-sum sequence of length $|S|$ such that $|S|\leq 3$ or $|S| \geq \lfloor\frac{n}{2}\rfloor+2$, then
ind(S) = 1 (see \cite{CFS:99,SC:07,Y:09}). In contrast to that, it was shown that for each $l \mbox{ with }  5 \leq l \le \lfloor\frac{n}{2}\rfloor + 1$, there is a minimal zero-sum sequence $S$ of length $|S| = l$ with $\ind(S)\geq 2$ (\cite{SC:07, Y:09})  and that the
same is true for $l = 4$ and $\gcd(n, 6) \ne 1$ (\cite{P:04}). In  two recent papers \cite{LPYZ:10,LP:12},  the  authors proved that $\ind(S)=1$ if $|S|=4$ and $\gcd(n, 6) = 1$ when $n$ is a prime power or a product of two prime powers with some restriction. However,  the general case is still open.

\bigskip

Let $S=(n_1g)\cdot\ldots\cdot(n_lg)$ be a minimal zero-sum sequence of length $l$ over $G$. Suppose that there exist an element $ag\in S $ and two elements $xg, yg \in G$ such that $xg+yg = ag$ and  $T=S(ag)^{-1}(xg)(yg)$ is a minimal zero-sum sequence of length $l+1$. Clearly $\ind(S) \le \ind(T)$ as $||S||_g \leq ||T||_g$ for all $g \in G$ with $G=\langle g \rangle$. In this case, the investigation of the index of a minimal zero-sum sequence of length 4 can be transformed into the investigation of the index of a minimal zero-sum sequence of length 5. In order to further investigate the index of a general minimal zero-sum sequence of length 4, it is helpful to determine the index of certain minimal zero-sum sequences of length 5. Little is known about the index of a minimal zero-sum sequence  over $G$ of length 5. It is routine to check that  if $S$ is a minimal zero-sum sequence over $G$ of length 5, then $1\le \ind(S) \le 2$.   Let $\mathsf h(S)$ be the maximal repetition  of an element in $S$. Suppose that  $|G|$ is a prime.  It is shown in Proposition \ref{index 1} that if $\mathsf h(S) \geq 3$, then $\ind(S)=1$. If $\mathsf h(S)=2$, there exist minimal zero-sum sequences $S$ of length 5 with $\ind(S)=2 $ (see Propositions \ref{index 2} and \ref{small p with index 2} below for details). The main purpose of the present paper is to determine the index of  a minimal zero-sum sequence $S$ over $G$ of length 5 with $\mathsf h(S) \geq 2$. Our main result is  as follows.

\begin{theorem}\label{main theorem}
Let $G$ be a cyclic group of order $p$ for some prime $p \ge 31$, and let $S\in \mathcal F(G)$ be a minimal zero-sum sequence of length $|S|=5$  with $\mathsf h(S) \geq 2$. Then $\ind(S)\in \{1, 2\}$, and  $\ind(S)=2$  if and only if  $S=g^2(\frac{p-1}{2}g)(\frac{p+3}{2}g)((p-3)g)$ for some $g\in G$.
\end{theorem}

We remark that Theorem \ref{main theorem} together with Propositions \ref{index 1} and  \ref{small p with index 2}  determines completely  the index of every minimal zero-sum sequence $S$ of length $5$ with $\mathsf h(S) \geq 2$. However, the remaining case when $\mathsf h(S) =1$ is much more complicated and  $\ind(S)$ is not yet determined.

\bigskip
\section {Preliminaries}
\medskip

We first prove some preliminary results which will be needed in the next section.  Let $G$ be a cyclic group of order $n$. Suppose that  $S=(n_1g)\cdot \, \ldots \,\cdot (n_lg)$ for some $g\in G$. Let  $ \|S\|_g' =\ord(g) \|S\|_g=\sum_{ i=1}^l n_i \in \N_0$ and denote by $|x|_n$ the least positive residue of $x$ modulo $n$, where $n \in \N$ and $x\in \Z$.  Let $mS$ denote the sequence $(mn_1g)\cdot \, \ldots \,\cdot (mn_lg)$. If $\ord(g)=n$, then  $ mS=(|mn_1|_ng)\cdot \, \ldots \,\cdot (|mn_l|_ng).$  We note that if $\gcd(n, m) = 1$, then the multiplication by $m$ is a group automorphism of $G$ and hence $\ind(S)=\ind(mS)$.

\begin{proposition}\label{index 1}
Let   $G$ be a cyclic group of  prime order $p$ and  $S\in \mathcal F(G)$ be a minimal zero-sum sequence of length $5$. If $\mathsf h(S) \ge 3$, then $\ind(S)=1$.
\end{proposition}

\begin{proof}
Suppose that  $S=(n_1g)\cdot \, \ldots \,\cdot (n_5g)$ for some $g\in G$ and $ 1 \le n_1 \le  \cdots \le n_5 <  p$. Since $\mathsf h (S) \ge 3$, without loss of generality we may assume that $n_1=n_2=n_3=1$. Since $S$ is a minimal zero-sum sequence, we have that  $\|S\|_g'=3+n_4+n_5 < 2p$. Therefore $\ind(S)=1$.
\end{proof}

\begin{proposition}\label{index 2}
Let   $G$ be a cyclic group of  prime order $p \ge 5$. If $S=g^2\cdot (\frac{p-1}{2}g)\cdot (\frac{p+3}{2}g)\cdot ((p-3)g)\in \mathcal{F}(G)$, then $\ind(S)=2$.
\end{proposition}

\begin{proof} Since $\|S\|_g'=2p,$ it suffices to show for any $m\in [1, p-1],$ we have $\|mS\|_g'>p$. Then $\ind(S)=2$.

First assume that  $m=2k$. Then $|m(\frac{p-1}{2})|_p=|kp-k|_p=p-k$. Note that $|m(\frac{p+3}{2})|_p \geq 1$ and $|m(p-3)|_p \geq 1$.  Therefore,  $\|mS\|_g'\geq 2k+2k+(p-k)+1+1>p$ and we are done.

Next suppose that  $m=2k+1$, then $2k+1\le p-2$ and thus $k\le \frac{p-3}{2}$. Hence $$|(2k+1)(\frac{p-1}{2})|_p=|kp-k+\frac{p-1}{2}|_p=\frac{p-1}{2}-k.$$

If $k< \frac{p-3}{6}$, then $|(2k+1)(\frac{p+3}{2})|_p=\frac{p+3}{2}+3k,$ $|(2k+1)(p-3)|_p=p-6k-3$.  Therefore, $\|mS\|_g'=(2k+1)+(2k+1)+(\frac{p-1}{2}-k)+(\frac{p+3}{2}+3k)+(p-6k-3)=2p > p$.

If $\frac{p-3}{6}< k<\frac{2p-3}{6}$, then $|(2k+1)(\frac{p+3}{2})|_p=3k-\frac{p-3}{2},$ $|(2k+1)(p-3)|_p=2p-6k-3,$ so  $\|mS\|_g'=4k+2+(\frac{p-1}{2}-k)+(3k-\frac{p-3}{2})+(2p-6k-3)=2p > p$.

If $\frac{2p-3}{6}\le k\le \frac{p-3}{2}$, then $|(2k+1)(\frac{p+3}{2})|_p=3k-\frac{p-3}{2},$ $|(2k+1)(p-3)|_p=3p-6k-3,$ so  $\|mS\|_g'=4k+2+(\frac{p-1}{2}-k)+(3k-\frac{p-3}{2})+(3p-6k-3)=3p > p$.

This completes the proof.
\end{proof}

\medskip
\begin{proposition}\label{small p with index 2}
Let   $G= \langle g \rangle$ be a cyclic group of  order $p$ for some prime  $p \in [5,59]$,  and let $S=g^2(x_1g)(x_2g)(x_3g)$ be a minimal zero-sum sequence over $G$, where $2 \le x_1 \le x_2 \le x_3 \le p-3$. Then $\ind(S)=2$ if and only if one of the following conditions holds.
\begin{enumerate}

\item[(1).]  $x_1=\frac{p-1}{2}, x_2=\frac{p+3}{2}, x_3=p-3$.

\item[(2).]  $p = 17$ and $x_1=8, x_2= 11, x_3=13$.

\item[(3).]  $p = 19$ and $x_1=6, x_2= 14, x_3=16$.

\item[(4).]  $p = 19$ and $x_1=9, x_2= 12, x_3=15$.

\item[(5).]  $p = 23$ and $x_1=11, x_2= 15, x_3=18$.

\item[(6).]  $p = 23$ and $x_1=9, x_2= 15, x_3=20$.

\item[(7).]  $p = 29$ and $x_1=14, x_2= 19, x_3=23$.

\end{enumerate}

\end{proposition}

\begin{proof}
It is routine to check the proposition holds and we omit the proof here.
\end{proof}

\begin{lemma}\label{special case}
Let   $G= \langle g \rangle$ be a cyclic group of  prime order $p \ge 5$, and let $S=g^2(cg)((p-b)g)((p-a)g)$ be a minimal zero-sum sequence over $G$ with $2+c=a+b$ and $2 < a \le b < c < \frac{p}{2}.$ Then $\ind(S)=1$ if one of the following conditions holds.
\begin{enumerate}

\item[(1).]  $a=4, \ b= 6, \ c=8$ and $p > 17$.

\item[(2).]  $a=4, \ b= 7, \ c=9$ and $p > 19$.

\item[(3).]  $a=3, \ b= 4, \ c=5$ and $p > 15$.

\item[(4).]  $a=3, \ b= 5, \ c=6$ and $p > 24$.

\end{enumerate}

\end{lemma}

\begin{proof}
(1). Suppose $p=6m+t$, where $1 \le t \le 5$. Then $\gcd(m, p)=1$ and $\|mS\|_g'= \frac{p-t}{6} + \frac{p-t}{6} + \frac{2p-8t}{6} + t +\frac{2p+4t}{6}=p$. Therefore, $\ind(S)=1$.

(2). Suppose $p=7m+t$, where $1 \le t \le 6$. Then $\gcd(m, p)=1$ and $\|mS\|_g'= \frac{p-t}{7} + \frac{p-t}{7} + \frac{2p-9t}{7} + t +\frac{3p+4t}{7}=p$. Therefore, $\ind(S)=1$.

(3). Suppose $p=4m+t$, where $1 \le t \le 3$. Then $\gcd(m, p)=1$ and $\|mS\|_g'= \frac{p-t}{4} + \frac{p-t}{4} + \frac{p-5t}{4} + t +\frac{p+3t}{4}=p$. Therefore, $\ind(S)=1$.

(4). Suppose $p=5m+t$, where $1 \le t \le 4$. Then $\gcd(m, p)=1$ and $\|mS\|_g'= \frac{p-t}{5} + \frac{p-t}{5} + \frac{p-6t}{5} + t +\frac{2p+3t}{5}=p$. Therefore, $\ind(S)=1$.
\end{proof}

\bigskip
\section{Proof of main theorem}

\medskip

In this section we determine the index of every minimal zero-sum sequence $S$ of length 5 over a cyclic group of a prime order with $\mathsf h(S) \geq 2$. Let   $G= $ be a cyclic group of  prime order $p \ge 31$  and $S\in \mathcal F(G)$ be a minimal zero-sum sequence of length $5$. We will show that $\ind(S)=1$ except in the case that  $S=g^2(\frac{p-1}{2}g)(\frac{p+3}{2}g)((p-3)g)$  for some $g\in G$.

\smallskip

According to Proposition~\ref{index 1}, we may always assume that $\mathsf h(S) = 2$. Since $p$ is a prime, there exists $g\in G$ such that $S=g^2(x_1g)(x_2g)(x_3g)$, where $1< x_1 \le x_2 \le x_3 < p-2$. This implies that $1+1+x_2+x_2+x_3 < 3p$. If $1+1+x_1+x_2+x_3=p$, then $\ind(S)=1$. So we may assume that $1+1+x_1+x_2+x_3 =2p$. If $x_3 >x_2 > x_1 > \frac{p}{2}$, then $ \|2S\|_g'= 2 + 2 + (2x_1-p)+(2x_2-p)+(2x_3-p)=p $, and hence $\ind(S)=1$. So we may assume that $x_1 < \frac{p}{2}$. Clearly $x_2> \frac{p}{2}$, otherwise $1+1+x_1+x_2+x_3 < 1+1+ \frac{p}{2}+ \frac{p}{2}+x_3< 2p$, yielding a contradiction. Let $c=x_1, b = p-x_2,$ and $ a=p-x_3$. Then we can write $S$ in the form
\begin{equation}\label{S}
S=g^2(cg)((p-b)g)((p-a)g),
\end{equation}
where $2+c=a+b$ and $2 < a \le b < c < \frac{p}{2}.$


By Proposition \ref{index 2}, it suffices to show that if $a\ne 3 $ or $c \ne \frac{p-1}{2}$, then $\ind(S)=1$.  To do so, we will find $k$ and $m$ such that
\begin{equation}\label{*}
\frac{kp}{c} \le m <\frac{kp}{b}, \ \gcd(m,p)=1, \ 1 \le k \le b,  \mbox{ and }   ma < p.
\end{equation}
Then $\|mS\|_g'\le m +m + (mc-kp)+(kp - mb)+(p-ma)=p$, and thus $\ind(S)=1$.

Let $k_1$ be the largest positive integer such that $\lceil \frac{(k_1-1)p}{c} \rceil = \lceil \frac{(k_1-1)p}{b} \rceil$ and $\frac{k_1p}{c} \le m_1 < \frac{k_1p}{b}$. Since $\frac{bp}{c} \le p-1 < p = \frac{bp}{b}$ and $\frac{tp}{b} - \frac{tp}{c} = \frac{t(c-b)p}{bc} > 2$ for all $ t \ge b$, such integer $k_1$ always exists and $ k_1 \le b$.  Since $\lceil \frac{(k_1-1)p}{c} \rceil = \lceil \frac{(k_1-1)p}{b} \rceil$, we have
\begin{equation}\label{k1 is minimal}
1 > \frac{(k_1-1)p}{b}- \frac{(k_1-1)p}{c} = \frac{(k_1-1)p(c-b)}{bc} =\frac{(k_1-1)p(a-2)}{bc}.
\end{equation}

\medskip

Throughout this section we always assume that $S$ and $k_1$  are defined as above. We first handle some special cases, and then provide a proof of the main theorem.

\medskip

In terms of Proposition~\ref{small p with index 2}, from now on we may always assume that $p\geq 31$.
\medskip

\begin{lemma}\label{special case 1}
If $S$ is  a minimal zero-sum sequence such that $\ k_1 \ge 2$, $\ 3 < \frac{p}{c} < \frac{p}{b} < 4$, $\ a=3, \ b=3k_1-1$ and $c=3k_1$, then $\ind(S)=1$.
\end{lemma}

\begin{proof}
Suppose that $p=3b+b_0=9k_1-3+b_0$. Then $b_0 \not\equiv 0 \pmod 3$. Since $\frac{p}{c} > 3$,  we infer that $3 < b_0 < b = 3k_1-1$. By \eqref{k1 is minimal} we have $1 > \frac{(k_1-1)(9k_1-3+b_0)}{(3k_1-1)(3k_1)}.$ Hence $b_0k_1  - 9k_1 +3 - b_0 < 0$. If $b_0 \ge 15$, then $0 > b_0(k_1 -1)  - 9k_1 +3 \ge 15k_1-15-9k_1+3 \ge 0$, yielding a contradiction. Hence we must have $4 \le b_0 \le 14$ and $\gcd(b_0, 3) = 1$.

If $11 \le b_0 \le 14$, then $0 > b_0(k_1 -1)  - 9k_1 +3 = 11k_1-11-9k_1+3 = 2k_1-8 $ and thus $k_1 \le 3$. Since $11 \le b_0 < 3k_1-1$, we infer that $k_1 > 4$, a contradiction.

If $b_0 = 10$, then $0 > b_0(k_1 -1)  - 9k_1 +3 = 10k_1-10-9k_1+3 = k_1-7 $ and thus $k_1 < 7$. Since $10 = b_0 < 3k_1-1$, we infer that $k_1 \ge 4$. If $k_1 \le 5$, then $p\le 52$, the result follows from Lemma \ref{small p with index 2}. If $k_1=6$, then $p=61$. Since $\frac{2p}{c} < 7 < \frac{2p}{b}$ and $7a=21< p$, Equation~\eqref{*} holds and we are done.

If $b_0=8$, then $\frac{p}{c}= 3+ \frac{5}{3k_1}$ and $\frac{p}{b} = 3 + \frac{8}{3k_1-1}$. By the definition of $k_1$, we have $\lceil \frac{(k_1-1)p}{c} \rceil = \lceil \frac{(k_1-1)p}{b} \rceil$. Since $\frac{(k_1-1)p}{c} = 3k_1-3 + \frac{5(k_1-1)}{3k_1} < 3k_1-3+2, $ we have $\frac{(k_1-1)p}{b} = 3k_1-3 + \frac{8(k_1-1)}{3k_1-1} < 3k_1-3+2,$ then $k_1=2$. But $8=b_0 < 3k_1-1=5$, yielding a contradiction.

If $b_0=7$, then $\frac{p}{c}= 3+ \frac{4}{3k_1}$ and $\frac{p}{b} = 3 + \frac{7}{3k_1-1}$. As above since $\frac{(k_1-1)p}{c} = 3k_1-3 + \frac{4(k_1-1)}{3k_1} < 3k_1-3+2, $ we have $\frac{(k_1-1)p}{b} = 3k_1-3 + \frac{7(k_1-1)}{3k_1-1} < 3k_1-3+2,$ so $k_1 \le 4$. Since $7 = b_0 < 3k_1-1$, we infer that $k_1 \ge 3$. If $k_1=3$, then $p=31$, the lemma follows from Lemma \ref{small p with index 2}. If $k_1=4$, then $p=40$, a contradiction to that $p$ is prime.

If $b_0=5$, then $\frac{p}{c}= 3+ \frac{2}{3k_1}$ and $\frac{p}{b} = 3 + \frac{5}{3k_1-1}$. As above since $\frac{(k_1-1)p}{c} = 3k_1-3 + \frac{2(k_1-1)}{3k_1} < 3k_1-3+1, $ we have $\frac{(k_1-1)p}{b} = 3k_1-3 + \frac{5(k_1-1)}{3k_1-1} < 3k_1-3+1,$ so $k_1 <2$, yielding a contradiction.

If $b_0=4$, then $\frac{p}{c}= 3+ \frac{1}{3k_1}$ and $\frac{p}{b} = 3 + \frac{4}{3k_1-1}$. As above since $\frac{(k_1-1)p}{c} = 3k_1-3 + \frac{k_1-1}{3k_1} < 3k_1-3+1, $ we have $\frac{(k_1-1)p}{b} = 3k_1-3 + \frac{4(k_1-1)}{3k_1-1} < 3k_1-3+1,$ so $k_1 =2$. Therefore $p=19 < 31$, yielding a contradiction.
\end{proof}

\medskip

\begin{lemma}\label{special case 2}
There exists no minimal zero-sum sequence $S$ such that $\ k_1 \ge 2$, $\ 3 < \frac{p}{c} < \frac{p}{b} < 4$, $\ a=3, \ b=3k_1-2$ and $c=3k_1-1$.
\end{lemma}

\begin{proof}
Assume to the contrary that such $S$ exists. Suppose $p=3b+b_0=9k_1-6+b_0$. Then $b_0 \not\equiv 0 \pmod 3$. Since $\frac{p}{c} > 3$,  we infer that $3 < b_0 < 3k_1-2$. By \eqref{k1 is minimal} we have $1 > \frac{(k_1-1)(9k_1-6+b_0)}{(3k_1-2)(3k_1-1)}.$ Hence $b_0k_1  - 6k_1 +4 - b_0 < 0$. If $b_0 \ge 8$, then $0 > b_0(k_1 -1)  - 6k_1 +4 \ge 8k_1-8-6k_1+4 \ge 0$, yielding a contradiction. Hence we must have $4 \le b_0 \le 7$.

If $b_0 = 7$, then $0 > b_0(k_1 -1)  - 6k_1 +4 = 7k_1-7-6k_1+4 = k_1-3 $ and thus $k_1 =2$. Since $7 = b_0 < 3k_1-2$, we infer that $k_1 > 3$, a contradiction.

If $b_0=5$, then $\frac{p}{c}= 3+ \frac{2}{3k_1-1}$ and $\frac{p}{b} = 3 + \frac{5}{3k_1-2}$. By the definition of $k_1$, we have $\lceil \frac{(k_1-1)p}{c} \rceil = \lceil \frac{(k_1-1)p}{b} \rceil$. But $\frac{(k_1-1)p}{c} = 3k_1-3 + \frac{2(k_1-1)}{3k_1-1} < 3k_1-3+1 < 3k_1-3 + \frac{5(k_1-1)}{3k_1-2} = \frac{(k_1-1)p}{b},$ yielding a contradiction.

If $b_0=4$, then $\frac{p}{c}= 3+ \frac{1}{3k_1-1}$ and $\frac{p}{b} = 3 + \frac{4}{3k_1-2}$. As above we have $\frac{(k_1-1)p}{c} = 3k_1-3 + \frac{k_1-1}{3k_1-1} < 3k_1-3+1 < 3k_1-3 + \frac{4(k_1-1)}{3k_1-2} = \frac{(k_1-1)p}{b},$ yielding a contradiction.

In all cases, we have found contradictions. Thus such sequence $S$ does not exist.
\end{proof}

\medskip
\begin{lemma}\label{special case 3}
If $S$ is a minimal zero-sum sequence such that $\ k_1 \ge 5$, $\ 2 < \frac{p}{c} < \frac{p}{b} < 3$, $\ a=4, \ b=4k_1-1$ and $c=4k_1+1$, then $\ind(S)=1$.
\end{lemma}

\begin{proof}
Suppose $p=2b+b_0=8k_1-2+b_0$. Then $b_0 \equiv 1 \pmod 2$. Since $\frac{p}{c} > 2$,  we infer that $4< b_0 < 4k_1-1$. By \eqref{k1 is minimal} we have $1 > \frac{2(k_1-1)(8k_1-2+b_0)}{(4k_1-1)(4k_1+1)}.$ Hence $2b_0k_1  - 20k_1 +5 - 2b_0 < 0$. If $b_0 \ge 12$, then $0 > b_0(2k_1 -2)  - 20k_1 +5 \ge 24k_1-24-20k_1+5 \ge 0$, yielding a contradiction. Hence we must have $5\le b_0 \le 11$.

If $b_0 = 11$, then $0 > b_0(2k_1 -2)   - 20k_1 +5 = 22k_1-22-20k_1 +5 = 2k_1-17 $ and thus $k_1 \le 8$. If $k_1=8$, then $p=73, \ b= 31, \ c= 33$. Since $\frac{4p}{c} < 9 < \frac{4p}{b}$ and $9a=36<p$, we are done. If $k_1=7$, then $p=67,  \ b= 27, \ c= 29$. Since $\frac{3p}{c} < 7 < \frac{3p}{b}$ and $7a=28< p$, we are done. If $k_1=6$, then $p=57$, a contradiction to $p$ is prime. If $k_1=5$, then $p=49$, a contradiction again.

If $b_0=9$, then $\frac{p}{c}= 2+ \frac{5}{4k_1+1}$ and $\frac{p}{b} = 2 + \frac{9}{4k_1-1}$. By the definition of $k_1$, we have $\lceil \frac{(k_1-1)p}{c} \rceil = \lceil \frac{(k_1-1)p}{b} \rceil$. Since $\frac{(k_1-1)p}{c} = 2k_1-2 + \frac{5(k_1-1)}{4k_1+1} < 2k_1-2+2, $ we have $\frac{(k_1-1)p}{b} = 2k_1-2 + \frac{9(k_1-1)}{4k_1-1} < 2k_1-2+2,$ then $k_1< 7$. If  $k_1=6$, then $p=55$, a contradiction to that $p$ is prime. If $k_1=5$, then $p=47, \ b=19, \ c= 21$, the result follows from Lemma \ref{small p with index 2}.

If $b_0=7$, then $\frac{p}{c}= 2+ \frac{3}{4k_1+1}$ and $\frac{p}{b} = 2 + \frac{7}{4k_1-1}$. By the definition of $k_1$, we have $\lceil \frac{(k_1-1)p}{c} \rceil = \lceil \frac{(k_1-1)p}{b} \rceil$. But $\frac{(k_1-1)p}{c} = 2k_1-2 + \frac{3(k_1-1)}{4k_1+1} < 2k_1-2+1 <  2k_1-2 + \frac{7(k_1-1)}{4k_1-1} = \frac{(k_1-1)p}{b},$ yielding a contradiction.

If $b_0=5$, then $\frac{p}{c}= 2+ \frac{1}{4k_1+1}$ and $\frac{p}{b} = 2 + \frac{5}{4k_1-1}$. As above we have  $\frac{(k_1-1)p}{c} = 2k_1-2 + \frac{(k_1-1)}{4k_1+1} < 2k_1-2+1 < 2k_1-2 + \frac{5(k_1-1)}{4k_1-1}  = \frac{(k_1-1)p}{b},$ yielding a contradiction.
\end{proof}

\medskip
\begin{lemma}\label{special case 4}
There exists no minimal zero-sum sequence $S$ such that $\ k_1 \ge 5$, $\ 2 < \frac{p}{c} < \frac{p}{b} < 3$, $\ a=4, \ b=4k_1-2$ and $c=4k_1$.
\end{lemma}

\begin{proof}
Assume to the contrary that such $S$ exists. Suppose $p=2b+b_0=8k_1-4+b_0$. Then $b_0 \equiv 1 \pmod 2$. Since $\frac{p}{c} > 2$,  we infer that $4< b_0 < 4k_1-2$. By \eqref{k1 is minimal} we have $1 > \frac{2(k_1-1)(8k_1-4+b_0)}{(4k_1-2)(4k_1)}.$ Hence $b_0k_1  - 8k_1 + 4 - b_0 < 0$. If $b_0 \ge 9$, then $0 > b_0(k_1 -1)  - 8k_1 +4 \ge 9k_1-9-8k_1+4 \ge 0$, yielding a contradiction. Hence we must have $5\le b_0 \le 7$.

If $b_0=7$, then $\frac{p}{c}= 2+ \frac{3}{4k_1}$ and $\frac{p}{b} = 2 + \frac{7}{4k_1-2}$. By the definition of $k_1$, we have $\lceil \frac{(k_1-1)p}{c} \rceil = \lceil \frac{(k_1-1)p}{b} \rceil$. But $\frac{(k_1-1)p}{c} = 2k_1-2 + \frac{3(k_1-1)}{4k_1} < 2k_1-2+1 < 2k_1-2 + \frac{7(k_1-1)}{4k_1-1} = \frac{(k_1-1)p}{b},$ yielding a contradiction.

If $b_0=5$, then $\frac{p}{c}= 2+ \frac{1}{4k_1}$ and $\frac{p}{b} = 2 + \frac{5}{4k_1-2}$. As above  $\frac{(k_1-1)p}{c} = 2k_1-2 + \frac{(k_1-1)}{4k_1} < 2k_1-2+1 < 2k_1-2 + \frac{5(k_1-1)}{4k_1-2} = \frac{(k_1-1)p}{b},$ yielding a contradiction.
\end{proof}

\medskip
\begin{lemma}\label{special case 5}
There exists no minimal zero-sum sequence $S$ such that $\ k_1 \ge 5$, $\ 2 < \frac{p}{c} < \frac{p}{b} < 3$, $\ a=4, \ b=4k_1-3$ and $c=4k_1-1$.
\end{lemma}

\begin{proof}
Assume to the contrary that such $S$ exists. Suppose $p=2b+b_0=8k_1-6+b_0$. Then $b_0 \equiv 1 \pmod 2$. Since $\frac{p}{c} > 2$,  we infer that $4< b_0 < 4k_1-3$. By \eqref{k1 is minimal} we have $1 > \frac{2(k_1-1)(8k_1-6+b_0)}{(4k_1-3)(4k_1-1)}.$ Hence $2b_0k_1  - 12k_1 + 9 - 2b_0 < 0$. If $b_0 \ge 7$, then $0 > b_0(2k_1 -2)  - 12k_1 +9 \ge 14k_1-14-12k_1+9 \ge 0$, giving a contradiction. Hence we must have $b_0 =5$.

If $b_0=5$, then $\frac{p}{c}= 2+ \frac{1}{4k_1-1}$ and $\frac{p}{b} = 2 + \frac{5}{4k_1-3}$. By the definition of $k_1$, we have $\lceil \frac{(k_1-1)p}{c} \rceil = \lceil \frac{(k_1-1)p}{b} \rceil$. But $\frac{(k_1-1)p}{c} = 2k_1-2 + \frac{(k_1-1)}{4k_1-1} < 2k_1-2+1 < 2k_1-2 + \frac{5(k_1-1)}{4k_1-3} = \frac{(k_1-1)p}{b},$ yielding a contradiction.
\end{proof}

\medskip
\begin{lemma}\label{special case 6}
If $S$ is a minimal zero-sum sequence such that $\ k_1 \ge 5$, $\ 2 < \frac{p}{c} < \frac{p}{b} < 3$, $\ a=3, \ b=2k_1+k_0$ and $c=2k_1+k_0+1< \frac{p-1}{2}$, where $0 \le k_0 \le k_1-1$, then $\ind(S)=1$.
\end{lemma}

\begin{proof}
We will show that there exist $x, y \in [1, \lfloor\frac{b}{3}\rfloor]$ such that $\frac{p}{c}  < 2+ \frac{x}{y} < \frac{p}{b}$. Then $(2y+x)a < \frac{yp}{b} \time 3 \le p$ and we are done.

Suppose $p=2b+b_0$, where $1 \le b_0 \le b-1$. Since $p$ is prime, we infer that $b_0 \equiv 1 \pmod 2$. Note that $c=b+1$. It suffices to show there exist $x, y \in [1, \lfloor\frac{b}{3}\rfloor]$ such that $\frac{b_0-2}{b+1}  < \frac{x}{y} < \frac{b_0}{b}$.

\medskip
\noindent {\bf Case 1.} $b \equiv 0 \pmod 3$. Since $p$ is prime, we infer that $b_0 \not\equiv 0 \pmod 3$. Suppose $b=3s$.

If  $b_0=3t+1$, then let $x=t$ and $y=s$. We infer that $\frac{3t-1}{3s+1} < \frac{t}{s} < \frac{3t+1}{3s}$, and we are done.

If  $b_0=3t+2$, then let $x=t$ and $y=s$. We infer that $\frac{3t}{3s+1} < \frac{t}{s} < \frac{3t+2}{3s}$, and we are done.

\medskip
\noindent {\bf Case 2.} $b \equiv 1 \pmod 3$. Since $p$ is prime, we infer that $b_0 \not\equiv 1 \pmod 3$.  Suppose $b=3s+1$.

First assume that $b_0=3t  \equiv 1 \pmod 2$. Since $c = b+1 < \frac{p-1}{2}= b + \frac{b_0-1}{2}$, we infer that $b_0 > 3$ and thus $t \ge 3$. If $s < 2t -2 $, then  let $x=t-1$ and $y=s$. We infer that $\frac{3t-2}{3s+2} < \frac{t-1}{s} < \frac{3t}{3s+1}$, and we are done. Next assume that $s \ge 2t-2$. Choose $y= s- \lceil\frac{s-2t+3}{3t-2}\rceil$ and $x=t-1$. We will show that $\frac{3t-2}{3s+2} < \frac{t-1}{y} < \frac{3t}{3s+1}$. Since $y= s- \lceil\frac{s-2t+3}{3t-2}\rceil \le s - \frac{s-2t+3}{3t-2} = \frac{ 3st -3s + 2t - 3 }{3t-2} < \frac{(t-1)(3s+2)}{3t-2}$, we have $\frac{3t-2}{3s+2} < \frac{t-1}{y}$. Since $t \ge 3$ and $s \ge 2t-2$, we infer that $\frac{ 3st -3s -t }{3t-2} > \frac{(t-1)(3s+1)}{3t}$. Since $y= s- \lceil\frac{s-2t+3}{3t-2}\rceil \ge s - \frac{s-2t+3 + 3t-3}{3t-2} = \frac{ 3st -3s -t }{3t-2} > \frac{(t-1)(3s+1)}{3t}$, we have $\frac{t-1}{y} < \frac{3t}{3s+1}$, and we are done.

Now assume that $b_0=3t+2$. Let $x=t$ and $y=s$. We infer that $\frac{3t}{3s+2} < \frac{t}{s} < \frac{3t+2}{3s+1}$, and we are done.

\medskip
\noindent {\bf Case 3.} $b \equiv 2 \pmod 3$. Since $p$ is prime, we infer that $b_0 \not\equiv 2 \pmod 3$.  Suppose $b=3s+2$.

\medskip
\noindent {\bf Subcase 3.1.} $b_0 \equiv 0 \pmod 3$. Suppose  $b_0=3t$. Recall that $b_0=3t  \equiv 1 \pmod 2$. Since $c = b+1 < \frac{p-1}{2}= b + \frac{b_0-1}{2}$, we infer that $b_0 > 3$ and thus $t \ge 3$.
If $s < 3t -3 $, then  let $x=t-1$ and $y=s$. We infer that $\frac{3t-2}{3s+3} < \frac{t-1}{s} < \frac{3t}{3s+2}$, and we are done. Next assume that $s \ge 3t-3$. Choose $y= s- \lceil\frac{s-3t+4}{3t-2}\rceil$ and $x=t-1$. We will show that $\frac{3t-2}{3s+3} < \frac{t-1}{y} < \frac{3t}{3s+2}$. Since $y= s- \lceil\frac{s-3t+4}{3t-2}\rceil \le s - \frac{s-3t+4}{3t-2} = \frac{ 3st -3s + 3t - 4 }{3t-2} < \frac{(t-1)(3s+3)}{3t-2}$, we have $\frac{3t-2}{3s+3} < \frac{t-1}{y}$. Since $t \ge 3$ and $s \ge 3t-3$, we infer that $\frac{ 3st -3s -1 }{3t-2} > \frac{(t-1)(3s+2)}{3t}$. Since $y= s- \lceil\frac{s-3t+4}{3t-2}\rceil \ge s - \frac{s-3t+4 + 3t-3}{3t-2} = \frac{ 3st -3s -1 }{3t-2} > \frac{(t-1)(3s+2)}{3t}$, we have $\frac{t-1}{y} < \frac{3t}{3s+2}$, and we are done.

\medskip
\noindent {\bf Subcase 3.2.} $b_0 \equiv 1 \pmod 3$. Suppose  $b_0=3t+1$. Recall that $b_0=3t +1  \equiv 1 \pmod 2$. Hence $t \equiv 0 \pmod 2$.

If $s > 2t  $, then  let  $x=t$ and $y=s$. We infer that $\frac{3t-1}{3s+3} < \frac{t}{s} < \frac{3t+1}{3s+2}$, and we are done.

If $s <  \frac{3t-3}{2}$, then  let  $x=t-1$ and $y=s$. We infer that $\frac{3t-1}{3s+3} < \frac{t-1}{s} < \frac{3t+1}{3s+2}$, and we are done.

Next assume that $\frac{3t-3}{2} \le s \le 2t$.

If $t >5$, then let $x=t-1$ and $y=s-1$. We infer that $\frac{3t-1}{3s+3} < \frac{t-1}{s-1} < \frac{3t+1}{3s+2}$, and we are done. If $t \le 5$, we have $t= 2$ or $4$.

If $t=2$, then $b_0=7$. Since $\frac{3}{2} \le s \le 4$, we have $2 \le s \le 4$. If $s \le 3$, then $b \le 11$ and $p \le 29$, yielding a contradiction to $p \ge 31$. If $s=4$, then $b=14$ and $p=35$, yielding a contradiction to that $p$ is prime.

If $t=4$, then $b_0=13$. Since $\frac{9}{2} \le s \le 8$, we have $5 \le s \le 8$.  If $s=5$, then $b=17$ and $p=47$, so the results follows from Lemma \ref{small p with index 2}. If $s=6$, then $b=20$ and $p=53$, so the results follows from Lemma \ref{small p with index 2}. If $s=7$, then $b=23$ and $p=59$, so the results follows from Lemma \ref{small p with index 2}. If $s=8$, then $b=26$ and $p=65$, yielding a contradiction to that $p$ is prime.
\end{proof}

\medskip

We are now in the position to prove the main theorem.
\medskip

\noindent {\bf Proof of Theorem \ref{main theorem}}

\medskip
We divide the proof according to the following three cases.
\medskip

\noindent {\bf Case 1.} $\lceil \frac{p}{c} \rceil < \lceil \frac{p}{b} \rceil$. Suppose that $\lceil \frac{p}{c} \rceil = m< \frac{p}{b}$. Let $k=1$. Then $ma \le  mb < p$, and we are done.

\medskip
\noindent {\bf Case 2.} $\lceil \frac{p}{c} \rceil = \lceil \frac{p}{b} \rceil$ and $k_1\le \frac{b}{a}$. Suppose $\lceil \frac{k_1p}{c} \rceil = m< \frac{k_1p}{b}$. Let $k=k_1$. Then $ma \le  m \frac{b}{k_1} < p$, and we are done.

\medskip
\noindent {\bf Case 3.} $\lceil \frac{p}{c} \rceil = \lceil \frac{p}{b} \rceil$ and $k_1 > \frac{b}{a}$. Then $k_1 \ge 2$.

If $a-2\ge \frac{b}{k_1}$, then $\frac{(k_1-1)p(a-2)}{bc}> \frac{2(k_1-1)}{k_1} \ge 1$, a contradiction to \eqref{k1 is minimal}. Hence we may assume that $a-2 < \frac{b}{k_1} < a$.

Now assume that $b=k_1\ell+k_0$, where $0 \le k_0 < k_1$. Then $a-2 \le \ell < \ell +1 \le a$.

\medskip
\noindent {\bf Subcase 3.1.} $a=\ell+1$. Then $c=a+b-2=(k_1+1)\ell + k_0-1$.

Suppose $\frac{p}{c} > 3$.  By \eqref{k1 is minimal} we have $1 > \frac{3(\ell-1)(k_1-1)}{k_1\ell+k_0} \ge \frac{3\ell k_1 - 3k_1 - 3\ell +3}{k_1\ell +k_1-1}.$ Hence $2\ell k_1 -3\ell -4k_1 +4 < 0$. This implies that $\ell =2$ or $\ell=3, \ k_1=2$.

If $\ell =2$, then $a=3$. If $\frac{p}{c} > 4$, then by \eqref{k1 is minimal} we have $1 > \frac{4(\ell-1)(k_1-1)}{k_1\ell+k_0} = \frac{4 k_1 - 4}{2k_1 +k_0}.$ Hence $2k_1  - k_0 -4 < 0$ and thus $k_1=2$. Hence $ b= 4$ or $5$. If $b=4$, then $c=5$, so the result follows from Lemma \ref{special case} (3). If $b=5$, then $c=6$, so the result follows from Lemma \ref{special case} (4). Next assume that $3< \frac{p}{c} < 4$. Since $\lceil \frac{p}{c} \rceil = \lceil \frac{p}{b} \rceil$ we have $3 < \frac{p}{c} < \frac{p}{b} < 4$. By \eqref{k1 is minimal} we have $1 > \frac{3(\ell-1)(k_1-1)}{k_1\ell+k_0} = \frac{3 k_1 - 3}{2k_1 +k_0}.$ Hence $k_1  - k_0 -3 < 0$ and thus $k_0=k_1-1$ or $k_1-2$. If $k_0=k_1-1$, then $b=3k_1-1$ and $c=3k_1$, so the result follows from Lemma \ref{special case 1}. If $k_0=k_1-2$, then $b=3k_1-2$ and $c=3k_1-1$, so it follows from Lemma \ref{special case 2} that this case is impossible.

If $\ell=3, \ k_1=2$, then $a=4$ and  $b=6$ or $7$. If $b=6$, then $c=8$, so the result follows from Lemma \ref{special case} (1). If $b=7$, then $c=9$, so the result follows from Lemma \ref{special case} (2).

Suppose that $3 > \frac{p}{c} >2$. Since $\lceil \frac{p}{c} \rceil = \lceil \frac{p}{b} \rceil$ we have $2 < \frac{p}{c} < \frac{p}{b} < 3$. By \eqref{k1 is minimal} we have $1 > \frac{2(\ell-1)(k_1-1)}{k_1\ell+k_0} \ge \frac{2\ell k_1 - 2k_1 - 2\ell +2}{k_1\ell +k_1-1}.$ Hence $\ell k_1 -2\ell -3k_1 +3 < 0$. This implies that $k_1=2$ or $k_1=3, \ \ell \le 5$ or $k_1=4, \ \ell \le 4 $ or $k_1 \ge 5, \ \ell \le 3$.

If $k_1=2$, then $k_0=0$ or $1$. Since $\frac{2p}{c} \le m_1 < \frac{2p}{b}$, we infer that $m_1=5$. If $5a < p$, we are done. Hence we may assume that $p < 5a = 5\ell +5$. Since $p > 2c = 6\ell +2k_0-2$, we have $5\ell +5 > 6\ell +2k_0-2$ and thus $\ell < 7$. Since $p \ge 31$, we infer that $\ell \ge 6$. Hence $a \ge 7$. Since $p < 5\ell +5 < 42$,  by Lemma \ref{small p with index 2} we have $\ind(S)=1$.

If $k_1=3$ and $\ell \le 5$, then $b=k_1 \ell + k_0 \le 17$. Hence $p < 3b \le 51$,  so the result follows from Lemma \ref{small p with index 2}.

If $k_1=4$ and $\ell \le 4$, then $b=k_1 \ell + k_0 \le 19$. Hence $p < 3b \le 57$, so the result follows from Lemma \ref{small p with index 2}.

If $k_1\ge 5$ and $\ell =3$, then $a=4$. By \eqref{k1 is minimal} we have $1 > \frac{2\time 2 \time (k_1-1)}{3k_1+k_0}.$ Hence $k_1  - k_0 -4 < 0$ and thus $k_0=k_1-1$ or $k_1-2$ or $k_1-3$. If $k_0=k_1-1$, then $b= 4k_1-1$ and $c= 4k_1+1$, so  the result follows from Lemma \ref{special case 3}. If $k_0=k_1-2$, then $b= 4k_1-2$ and $c= 4k_1$, yielding a contradiction (by Lemma \ref{special case 4}). If $k_0=k_1-3$, then $b= 4k_1-3$ and $c= 4k_1-1$,  yielding a contradiction (by Lemma \ref{special case 5}).

If $k_1\ge 5$ and $\ell =2$, then $a=3$. Therefore, the result follows from Lemma \ref{special case 6}.

\medskip
\noindent {\bf Subcase 3.2.} $a=\ell+2$. Then $c=a+b-2=(k_1+1)\ell + k_0$.

Suppose $\frac{p}{c} > 3$.  By \eqref{k1 is minimal} we have $1 > \frac{3\ell(k_1-1)}{k_1\ell+k_0} \ge \frac{3\ell k_1  - 3\ell }{k_1\ell +k_1-1}.$ Hence $2\ell k_1 -3\ell -k_1 +1 < 0$, which is impossible since $k_1\ge 2$ and $\ell \ge 1$.

Next assume that $3 > \frac{p}{c} >2$, by \eqref{k1 is minimal} we have $1 > \frac{2\ell(k_1-1)}{k_1\ell+k_0} \ge \frac{2\ell k_1  - 2\ell }{k_1\ell +k_1-1}.$ Hence $\ell k_1 -2\ell -k_1 +1 < 0$. This implies that $k_1=2$ or $\ell =1$.

If $k_1=2$, then $k_0=0$ or $1$. Since $\lceil \frac{p}{c} \rceil = \lceil \frac{p}{b} \rceil$ we have $2 < \frac{p}{c} < \frac{p}{b} < 3$. Since $\frac{2p}{c} \le m_1 < \frac{2p}{b}$, we infer that $m_1=5$. If $5a < p$, we are done. Hence we may assume that $p < 5a = 5\ell +10$. Since $p > 2c = 6\ell +2k_0$, we have $5\ell +10 > 6\ell +2k_0$ and thus $\ell < 10$. Since $p \ge 31$, we infer that $\ell \ge 5$. Hence $a \ge 7$. Since $p < 5\ell +10 < 60$,  by Lemma \ref{small p with index 2} we have $\ind(S)=1$.

If $\ell =1$, then $a=3, \ b= k_1+k_0, \ c= k_1+k_0+1$.  By \eqref{k1 is minimal} we have $1 > \frac{2\ell(k_1-1)}{k_1\ell+k_0} = \frac{2 k_1 - 2}{k_1 +k_0}.$ Hence $k_1  - k_0 -2 < 0$ and thus $k_0=k_1-1$. Then $b=2k_1-1$ and $c=2k_1$. Suppose $p=2b+b_0=4k_1-2+b_0$. Then $b_0$ is odd. Since $c < \frac{p-1}{2}$,  we infer that $3 < b_0 < 2k_1-1$. By \eqref{k1 is minimal} we have $1 > \frac{(k_1-1)(4k_1-2+b_0)}{(2k_1-1)(2k_1)}.$ Hence $b_0k_1  - 4k_1 +2 - b_0 < 0$. If $b_0 \ge 6$, then $0 > b_0(k_1 -1)  - 4k_1 +2 \ge 6k_1-6-4k_1+2 \ge 0$, a contradiction. Hence we must have $b_0 = 5$. Then $0 > b_0(k_1 -1)  - 4k_1 +2 = 5k_1-5-4k_1+2 = k_1-3$ and thus $k_1=2$. Then $p=11$, yielding a contradiction.

This completes the proof. \qed



\end{document}